\documentclass{elsarticle}

\usepackage{amsmath}
\usepackage{amssymb}
\usepackage[psamsfonts]{eucal}
\usepackage{fancyhdr}

\newtheorem{teo}{\sc Theorem}

\newdefinition{defi}{\sc Definition}
\newdefinition{rem}{\sc Remark}

\newdefinition{example}{Example}
\newproof{proof}{Proof}

\begin{document}

\begin{frontmatter}

\author[Branquinho]{Am\'\i lcar Branquinho}
\fnref{label2}
\ead{ajplb@mat.uc.pt}

\author[Foulquie Moreno]{Ana Foulqui\'e Moreno}
\corref{cor1} \cortext[cor1]{Corresponding author}
\ead{foulquie@ua.pt}

\author[Mendes]{Ana Mendes}
\fnref{label3}
\ead{aimendes@ipleiria.pt}

\address[Branquinho]{CMUC and Department of Mathematics, University of Coimbra, Apartado 3008, EC Santa Cruz, 3001-501 Coimbra, Portugal.}

\address[Foulquie Moreno]{CIDMA and  Department of Mathematics, University of Aveiro, 3810-193 Aveiro, Portugal.}

\address[Mendes]{Department of Mathematics, School of Technology and Management, Polytechnic Institute of Leiria,  2411-901  Leiria, Portugal.}

\fntext[label2,cor1,label3]{This work was supported by Center for Mathematics of the University of Coimbra, Center for Research and Development in Mathematics and Applications (University of Aveiro), and School of Technology and Management, Polytechnic Institute of Leiria, funded by the European Regional Development Fund through the program COMPETE and by the Portuguese Government through the FCT--Funda\c c\~ao para a Ci\^encia e a Tecnologia under the projects PEst-C/MAT/UI0324/2011 and PEst-C/MAT/UI4106/2011 with COMPETE number FCOMP-01-0124-FEDER-022690.}

\title{Matrix Orthogonal Polynomial in the theory of Full Kostant-Toda Systems}

\begin{abstract}
In this work we characterize a full Kostant-Toda system in terms of a family of matrix polynomials orthogonal with respect to a complex matrix measure.
In order to study the solution of this dynamical system we give explicit expressions for the Weyl function and we also obtain, under some conditions, a representation of the vector of linear functionals associated with this system.
\end{abstract}

\begin{keyword} 
Matrix orthogonal polynomials \sep linear functional \sep recurrence relation \sep operator theory \sep matrix Sylvester differential equations \sep full Kostant-Toda systems.
\MSC 33C45 \sep 39B42 \sep 47N20 \sep 34K99 \sep 42C05.
\end{keyword}
\end{frontmatter}

\section{Introduction} \label{sec:1}

Consider the following infinite system of differential equations
\begin{eqnarray} \label{sisteqdif}
                  \begin{cases}
                   \dot{a}_n=c_n-c_{n-2} \\
                    \dot{b}_n=c_n a_{n+1}-c_{n-1}a_n+d_n-d_{n-2} \\
                    \dot{c}_n=c_n(b_{n+1}-b_n)+d_{n}a_{n+2}-d_{n-1}a_n \\
                    \dot{d}_n=d_n(b_{n+2}-b_n)
                  \end{cases} , \ n \in {\mathbb{N}} \, ,
\end{eqnarray}
where the dot, ``$ \ \dot{} \ $'', means the differentiation with respect to $t \in \mathbb{R}$ and where we assume that $a_0=b_0=c_0=d_0=0$ and $c_1=0$.

Particular cases of this kind of dynamical system appear in the literature in different contexts. 
For example, in~\cite{Bogoyavlenskii1} and~\cite{Bogoyavlenskii2}, Bogoyavlenskii gave a classification of these dynamical systems which are a discrete generalization of a KdV equation and showed that such systems have interesting applications on Hamilton mechanics.
On the other hand in the work~\cite{Aptekarev} a particular case of a Bogoyavlenskii discrete dynamical system appears related with the study of spectral problems for higher order difference equations and it was studied using a method based on the analysis of the genetic sums formula for the moments of the associated operator. Also in~\cite{Sorokin} the authors studied another particular case of these dynamical systems investigating the spectral properties of the associated band operator.
More recently, in~\cite{Geng}, the authors propose to study systems of type~\eqref{sisteqdif} motivated by its bi-Hamiltonian structure, and the first two authors, studied the interpretation of some generalizations of the systems considered in the works mentioned before (cf.~\cite{doloresamilcarana2},~\cite{doloresamilcarana}, and~\cite{doloresamilcarana3}).

The system of equations~\eqref{sisteqdif} can be written as matrix Sylvester equation known as a {\it Lax pair}, $$\dot{J}=[J,J_{-}]=J\,J_{-}-J_{-}\,J \, , $$ where $J$ and $J_{-}$ are the operators which matrix representation is given respectively by
\begin{equation}\label{eq:J}
J=\left[
  \begin{matrix}
    b_1 & a_2 & 1 &  &  &  &   \\
    c_1 & b_2 & a_3 & 1 &  &  &     \\
    d_1 & c_2 & b_3& a_4 & 1 &  &      \\
     & d_2 & c_3& b_4 & a_5 & 1 &    \\
     &  & \ddots & \ddots & \ddots & \ddots & \ddots  \\
     \end{matrix}
\right] \, , \ \ J_{-}=\left[
  \begin{matrix}
    0 &  &  &  &    \\
    c_1 & 0 &  &    &    \\
    d_1 & c_2 & 0 &   &      \\
     & d_2 & c_3& 0  &    \\
     &   & \ddots & \ddots   &  \ddots \\
     \end{matrix}
\right].
\end{equation}

When $J$ is a bounded operator, then it is possible to define the {\it resolvent ope\-ra\-tor}, by
$$(zI -J)^{-1}=\sum_{n=0}^\infty \frac{
J^n}{z^{n+1}} \, , \ \ |z|>||J|| \, , $$
(see for example~\cite{Beckermann} and~\cite{Berezanskii}) and, the associated analytic function
\begin{equation}\label{weyl}
R_J(z)= \sum_{n=0}^\infty \frac{e_0^T
J^n e_0}{z^{n+1}} \, , \ \ |z|>||J|| \, ,
\end{equation} 
where $e_0= \left[
\begin{matrix}
I_{2\times 2}&  0_{2\times 2} & \cdots
\end{matrix}
\right]^T, $ known as the {\it Weyl function} associated with~$J$.

If we denote by $M_{ij}$ the $2 \times 2$ block matrices, of an infinite matrix M, formed by the entries of rows
$2i-1$, $2i$ and columns $2j-1$, $2j$, the matrix~$J^n$ can be written by blocks as
\begin{equation}\label{blockmatrix}
J^n=\left[
  \begin{matrix}
    J^n_{11} & J^n_{12} & \cdots   \\
    J^n_{11}  & J^n_{11}& \cdots    \\
    \vdots & \vdots & \ddots  \end{matrix}
\right].
\end{equation}
In this way, for each $n \in {\mathbb{N}}$, we have that the Weyl function~\eqref{weyl} can also be written in the following form
\begin{equation}\label{eq:J11}
R_J(z)= \sum_{n=0}^\infty \frac{J^n_{11}}{z^{n+1}} \, , \ \ |z|>||J|| \, .
\end{equation}

As a consequence of the Lax pair representation for~\eqref{sisteqdif} we observe that the operator theory establishes a connection between these systems and the theory of approximation. In fact, with the Lax pair representation for~\eqref{sisteqdif} we can associate to this system the Weyl function of $J$, $R_J$. On the other hand, we will see that the Weyl function of~$J$ and the complex measure of orthogonality, given by the generalized Markov function associated with the systems of matrix orthogonal polynomials defined by~$J$, are similar. Using the theory of matrix orthogonal polynomials we can obtain a representation of $R_J$ that gives a solution of the system~\eqref{sisteqdif}.

In that sense, we start by considering, for $d_{n-1} \neq 0$, $n=2, 3, \ldots$, the sequence of monic polynomials $\{p_n\}
$ satisfying the five term recurrence relation
\begin{eqnarray} \label{eq:100}
& &x^2 p_n=p_{n+2}+a_{n+2} p_{n+1}+b_{n+1} p_{n}+c_n p_{n-1}+d_{n-1} p_{n-2}\, , \ \ n \geq 2 \nonumber \\
& &p_{1}(x)=x-a_1 \, , \ \ a_1 \in {\mathbb{R}} \, , \ \ p_{-1}(x)=0 \, , \ \ p_{0}(x)=0 \, .
\end{eqnarray}

Notice that~\eqref{eq:100} can be written in matrix form
\begin{multline*} x^2 \left[
\begin{matrix}
   p_{2m} \\
   p_{2m+1}
\end{matrix}\right]=
\left[
\begin{matrix}
  1 & 0 \\
  a_{2m+3} & 1
\end{matrix}%
\right] \left[
\begin{matrix}
   p_{2m+2} \\
      p_{2m+3}
\end{matrix}\right]
 \\
+ \left[%
\begin{matrix}
  b_{2m+1} & a_{2m+2}\\
  c_{2m+1} & b_{2m+2}
\end{matrix}%
\right] \left[
\begin{matrix}
   p_{2m} \\
      p_{2m+1}
    \end{matrix}\right]
+\left[
\begin{matrix}
  d_{2m-1} & c_{2m}\\
  0 & d_{2m}
\end{matrix}
\right] \left[\begin{matrix}
   p_{2m-2} \\
      p_{2m-1}
\end{matrix}\right]
\,.
\end{multline*}
which can be read as a three term recurrence relation 
 \begin{equation}\label{eq:300}
 x^2 {{\mathcal{B}}}_m(x) = A_{m} {{\mathcal{B}}}_{m+1}(x) + B_m
{{\mathcal{B}}}_m(x) + C_m {{\mathcal{B}}}_{m-1}(x) \, , \ \ m \geq 1 \, ,
 \end{equation}
 where
${{\mathcal{B}}}_m = \left[\begin{matrix}
                                           p_{2m} & p_{2m+1}
                                                         \end{matrix}
                                                         \right]^T$,
\begin{eqnarray*} A_m =  \left[
\begin{matrix}
  1 & 0 \\
  a_{2m+3} & 1
\end{matrix}%
\right] \, , \ \  B_m = \left[
\begin{matrix}
  b_{2m+1} & a_{2m+2}\\
  c_{2m+1} & b_{2m+2}
\end{matrix}%
\right] \, , \ \ C_m = \left[
\begin{matrix}
  d_{2m-1} & c_{2m}\\
  0 & d_{2m}
\end{matrix}%
\right] \, ,
\end{eqnarray*}
for $m \geq 1$, with ${{\mathcal{B}}}_{-1}(x)=0_{2 \times 1}$ and ${{\mathcal{B}}}_{0}(x)=\left[
\begin{matrix}
1 & x - a_1
\end{matrix}
\right]^T $.

It was proved in~\cite{anamendes} that we can always write $\mathcal{B}_{m}$ in the matrix form
 \begin{eqnarray*}
\mathcal{B}_{m}(x)=V_{m}(x^2)\mathcal{P}_{0}(x),
\end{eqnarray*}
where $ V_{m}$ is a~$2\times 2$ matrix polynomial of degree $m$, $\mathcal{P}_{0}(x) = \left[
\begin{matrix}
1 & x
\end{matrix}
\right]^T$, and that $\{V_m\}
$ 
is defined by
\begin{eqnarray} \label{rrrV}
x  V_m(x) = A_{m+1} V_{m+1}(x)+ B_m V_m(x)+ C_mV_{m-1}(x) \, , \ \ m\geq 1 \, ,
\end{eqnarray}
with $V_{-1} = 0_{2 \times 2}$ and 
$V_0 = \left[  \begin{matrix}
              1 & 0 \\
              -a_1 & 1 
            \end{matrix}
          \right]
$.

From this it can be seen that $\{ V_m \}$ is a sequence of matrix polynomials, orthogonal with respect to a complex matrix measure that can be written in terms of the Weyl function $R_J$.  In the next section we shall see this statement in more detail.

Since, the matrix operator $J$ can also be written in terms of a block tridiagonal matrix of the form
\begin{equation} \label{J}
 J=\left[%
\begin{matrix}
  B_0 & A_0 & 0_{2 \times 2} &  \\
  C_1 & B_1 & A_1 & \ddots \\
  0_{2 \times 2} & C_2 & B_2 & \ddots \\
    & \ddots & \ddots& \ddots
\end{matrix}%
\right],
\end{equation}
then it is related to the matrix sequence of polynomials $\{ V_m \}
$ through the recurrence relation~\eqref{rrrV}. This block matrix, from now on, is said to be the $2 \times 2$ block
Jacobi matrix associated with the above matrix polynomial sequences.

Notice that the polynomials $p_n$ and $V_{m}$ depend on $t \in {\mathbb{R}}$, as well as the coefficients $a_n , b_n , c_n , d_n$ of the recurrence relations. For each~$t$,~$\{V_m \}
$ forms a matrix sequence of orthogonal polynomials. For sake of simplicity in the following we suppress the $t$-dependence.

One of our goals is to study the solutions of~\eqref{sisteqdif} in terms of the operator~$J$ and its asso\-cia\-ted to matrix polynomials $V_{m}(x)$, but first we should establish some known results about vector and matrix orthogonality.

At this point it is worth mentioning that~\eqref{sisteqdif} has a matrix interpretation in terms of the matrix coefficients $A_m$, $B_m$ and $C_m$ that appears in the recurrence relation~\eqref{eq:300}, i.e.
\begin{equation*} 
\begin{cases}
     \dot{A}_m=A_mD_{m+1}-D_mA_m\\
     \dot{B}_m=A_mC_{m+1}-C_mA_{m-1}+B_mD_m-D_mB_m\\
     \dot{C}_m=B_m C_m-C_mB_{m-1}+C_m D_{m-1}-D_mC_m
  \end{cases} \, , \ \ m=0,1, \ldots \, ,
\end{equation*}
with
$$D_m=\left[
        \begin{matrix}
          0 & 0 \\
          c_{2m+1} & 0 \\
        \end{matrix}
      \right] \, , \ \ m = 0 , 1 \ldots \, .
$$
 Also, these matrix coefficients contain the solution $\{a_n,b_n,c_n,d_n\}$, $n \in {\mathbb{N}}$ of the system~\eqref{sisteqdif} and, using the matrix orthogonality, we explicitly get the representations for them (cf. section~\ref{sec:2}).

In section~\ref{sec:2}, we present some known results about the vector and matrix or\-tho\-go\-na\-li\-ty.

In section~\ref{sec:3}, we study the solution of the dynamical system~\eqref{sisteqdif}. We show that the Weyl function associated to $J$ play a main role in the solution of this problem.

Finally, in section~\ref{sec:4}, we give explicit expressions for the Weyl function and we also obtain, under some conditions, a representation of the vector functionals associated with the system studied in section~\ref{sec:3}.

\section{Connection with vector orthogonality} \label{sec:2}

Let ${\mathbb{P}}$ be the linear space of polynomials with complex coefficients. Now, con\-si\-der the space of vector of polynomials ${\mathbb{P}}^2= \langle {\mathcal{P}}_j, \,
j \in {\mathbb{N}}\rangle $, where ${\mathcal{P}}_j=x^{2j}{\mathcal{P}}_0$ with ${\mathcal{P}}_0=\left[
\begin{matrix}
1 & x
\end{matrix}
\right]^T $,
and the space $ \mathcal{M} _{2\times 2}(\mathbb{C})$ of $2 \times 2$-matrices with complex entries. It is well known (see~\cite{anamendes}) that there exist a {\it vector of linear functionals} $\, {\mathcal{U}}= \left[
\begin{matrix}
u^1 & u^2
\end{matrix}
\right]^T$ defined in $(\mathbb{P}^{2})^*$, the linear space of vector linear functionals, here called {\it dual space}, acting in $ \mathbb{P}^{2}$
over~$ \mathcal{M} _{2\times 2}(\mathbb{C})$~such that
 \begin{eqnarray*}
\mathcal{U}(\mathcal{P}):= (\mathcal{U}.\mathcal{P} ^{T} )
^{T}=\left[
\begin{matrix}
\langle u^1, p_1 \rangle  & \langle u^2, p_1 \rangle  \\
\langle u^1, p_2 \rangle  & \langle u^2, p_2 \rangle
\end{matrix}
\right] \, ,
 \end{eqnarray*}
where ``$ . $'' means the symbolic product of the vectors
$ \mathcal{U}$ and $
\mathcal{P}^{T}$, where ${\mathcal{P}}^T= \left[
\begin{matrix}
p_1 & p_2
\end{matrix}
\right]^T$, $p_1 , p_2 \in {\mathbb{P}}$.
Notice that this definition is already known in a context of a vectorial interpretation of the multiple orthogonality (cf.~\cite{cotrimamilcar}).

It is easy to verify that ${\mathcal{U}}$ is linear, i.e., ${\mathcal{U}}$ satisfies ${\mathcal{U}}(A\,\mathcal{P}+B\,\mathcal{Q})=A\,\mathcal{U}(\mathcal{P})+B\,\mathcal{U}(\mathcal{Q})$ for $A$,
$B$ numerical matrices and $\mathcal{P}$, $\mathcal{Q}$ vector of polynomials and if $\displaystyle \widehat{A}(x) = \sum_{k=0}^{l}A_{k}\, x^{k}$ is a matrix polynomial we
can define the  \emph{left multiplication of $ \, \mathcal{U}$ by $\widehat{A}$}, denoted  by $\widehat{A}\, \mathcal{U}$,
as the vector of linear functionals such that
 \begin{equation*} 
 (\widehat{A} \, \mathcal{U} )(\mathcal{P} ):= (
  \widehat{A}\,\mathcal{U} . \mathcal{P}
^{T} )^{T}=\sum_{k=0}^{l} (x^{k}\,\mathcal{U} )( \mathcal{P}) \,
(A_{k} )^{T} \, .
\end{equation*}

The {\it Hankel matrices} associated with $\mathcal{U}$ are the matrices
 \begin{eqnarray*} 
 U_m =\left[%
 \begin{matrix}
  {\mathcal{U}}_0 & \cdots &  {\mathcal{U}}_m  \\
  \vdots & \ddots & \vdots \\
   {\mathcal{U}}_m & \cdots &  {\mathcal{U}}_{2m}
 \end{matrix}%
 \right] \, ,\ \ m \in {\mathbb{N}} \, ,
 \end{eqnarray*}
where ${\mathcal{U}}_j$ is the  $j$-th moment associated with the vector of linear functionals~${\mathcal{U}}$, i.e., ${\mathcal{U}}_j={\mathcal{U}}(x^{2j}{\mathcal{P}}_0)$.
${\mathcal{U}}$ is said to be \emph{quasi-definite} if all the leading principal submatrices of $U_m \, , \ m \in {\mathbb{N}}$, are non-singular.

A {\it vector sequence of polynomials} $\{{\mathcal{B}}_m\}
$, with degree of ${\mathcal{B}}_m$ equal to $m$, is 
{\it left-orthogonal} with respect to the vector of linear functionals $\mathcal{U}$~if
\begin{equation} \label{cotra1}
(x^{2k} {\mathcal{U}}) \left( {\mathcal{B}}_m
\right)=\Delta_m \delta_{k,m} \, , \ \  k=0,1,\ldots \, , \ m \, , \ \ m \in
{\mathbb{N}} \, ,
\end{equation}
with $\Delta_m$ a non-singular $2\times 2$ upper triangular matrix given~by
$$\Delta_m =C_m \,\cdots \,C_1 \, \Delta_0 \, ,
\ \ m \geq 1 \, , $$ where $\Delta_0$ is a~$2 \times 2$ non-singular matrix and $\{ C_m \}
$ is a sequence of non-singular upper triangular matrices.

Similarly, a {\it sequence of matrix polynomials} $\{G_m\}
$, with degree of $G_m$ equal to $m$, is 
{\it right-orthogonal} with respect to the vector of linear functionals~${\mathcal{U}}$ if it is {\it bi-orthogonal} with the vector sequence of polynomials $ \{ \mathcal{B}_{m} \}
$ to the vector of linear functionals~${\mathcal{U}}$, i.e., if
 \begin{eqnarray*}
 (   (  G_{n} (  x^2 )   ) ^{T}\mathcal{U} )
 (  \mathcal{B}_{m} )  = I_{2\times 2} \, \delta_{n,m} \, , \ \ n,m \in \mathbb{N} \, .
\end{eqnarray*}

In~\cite{anamendes}, necessary and sufficient conditions for the quasi-definiteness of ${\mathcal{U}}$, i.e., for the existence of a vector sequence of polynomials left-orthogonal with respect to the vector of linear functionals $\mathcal{U}$ were obtained.

These sequences of polynomials satisfy non-symmetric three term recurrence relations. So, if $\{{\mathcal{B}}_m\}
$ is a vector sequence of polynomials left-orthogo\-nal with respect to~${\mathcal{U}}$ 
where $\mathcal{B}_{m}(x)=V_{m}(x^2)\mathcal{P}_{0}(x)$, with $\mathcal{P}_{0}(x) = \left[
\begin{matrix}
1 & x
\end{matrix}
\right]^T$,
then there exist sequences of numerical matrices $ \{ A_m \}
$, $ \{ B_m \}
$, and $\{ C_m \}
$, with $A_m$ a non-singular lower triangular matrix and $C_m$ a non-singular upper triangular matrix, such that $\{ \mathcal{B}_m \}$ is defined by~\eqref{eq:300}
with
$ {\mathcal{B}}_{-1}(x) = 0_{1 \times 2}$ and $ {\mathcal{B}}_{0}(x)=M  {\mathcal{P}}_0(x) \, $,  for a fixed matrix, $M$. Moreover,
$\{ V_m \}$ is defined by~\eqref{rrrV}
with $V_{-1} = 0_{2 \times 2}$ and $V_0 = M $, and $\{ G_n \}$ is defined by
\begin{eqnarray*} \label{rrrG}
x  G_n (x) = G_{n+1}(x) C_{n+1} + G_n (x) B_n + G_{n-1}(x) A_{n-1} \, , \ \ n \geq 1 \, ,
\end{eqnarray*}
with $G_{-1} = 0_{2 \times 2}$ and $G_0 = \mathcal{U} \left( \mathcal{B}_0 \right)$.

These three term recurrence relation completely characterizes each type of orthogonality, i.e., there exist a Favard type theorem for each of these cases. Moreover, the coefficients of the three term recurrence relation can be expressed in terms of vector of linear functionals $\mathcal{U}$ or in terms of matrix measure associated with $\{V_m\}
$ or $\{G_n\}
$ (cf.~\cite{anamendes}).

Furthermore, {\it left and right vector orthogonality} is connected with {\it right and left matrix orthogonality} as will see bellow. So,  if we consider the formal series
$$
\frac{1}{z-x^2} = \sum_{n=0}^\infty
\frac{x^{2n}}{z^{n+1}} \, , \ \  |x^2|<|z| \, , 
$$ 
 the {\it generalized Markov matrix function}, ${\mathcal{F}}$, associated with~${\mathcal{U}}$  is defined~by
\begin{eqnarray*} 
{\mathcal{F}}(z) :=
\sum_{n=0}^\infty \frac{ \left( x^{2n} {\mathcal{U}} \right) \left( {\mathcal{P}}_0(x) \right) }{z^{n+1}}  
=
\left[
 \begin{matrix}
 \displaystyle \sum_{n=0}^\infty \frac{\langle  u^1 , x^{2n} \rangle}{z^{n+1}}  &  \displaystyle \sum_{n=0}^\infty \frac{ \langle u^2 , x^{2n} \rangle}{z^{n+1}}  \\ 
 \displaystyle \sum_{n=0}^\infty \frac{\langle  u^1 , x^{2n+1}\rangle }{z^{n+1}} & \displaystyle \sum_{n=0}^\infty \frac{\langle u^2 , x^{2n+1} \rangle }{z^{n+1}}
\end{matrix}%
\right],
\end{eqnarray*}
with $z$ such that $|x^2|<|z|$ for every  $x \in {\sf L}$ where
$  {\sf L} =   \cup_{j=1,2} \, \operatorname{supp} \, u^j \, $, and ${\mathcal{P}}_0(x) = \left[
\begin{matrix}
1 & x
\end{matrix}
\right]^T$.
\begin{teo}[{\rm cf.~\cite{anamendes}}]The matrix sequence $ \{ G_{n}\}
$ and the vector sequence of polynomials  $ \{ \mathcal{B}_{m} \}
$ are bi-orthogonal with respect to $\mathcal{U}$ if, and only if, the sequence of matrix orthogonal polynomials $ \{ G_{n} \}
$ and $ \{ V_{m} \}
$ are bi-orthogonal with respect to $\mathcal{F}$, i.e.,
 \begin{eqnarray*}
\frac{1}{2\pi i}\int_{C}V_{m} (  z )  \mathcal{F} ( z )  G_{n}
( z )  dz= I_{2 \times 2} \, \delta _{n,m} \, , \ \ n,m\in \mathbb{N} \, ,
 \end{eqnarray*}
where $C$ is a closed path in $\{z \in {\mathbb{C}}: |z| > |x^2|, x \in {\sf L} \}$ where is defined as before by
$  {\sf L} = \cup_{j=1,2} \, \operatorname{supp} \, u^j \, $.
\end{teo}

The sequences of matrix polynomials $\{V_m\}
$ and $\{G_m\}
$ presented here are orthogonal with respect to the complex matrix measure of orthogonality~$\mathcal{F}$. We note that~$\mathcal{F}$ is a complex matrix measure of orthogonality (cf.~\cite{Cberg}) which is not necessarily positive definite as in the orthornormal case considered in~(cf.~\cite{Dur93},~\cite{Dur96}) and that can be determined by a Markov type theorem as the reader can see in~\cite{anamendes2}. Moreover, the matrix sequence of polynomials $\{V_m\}
$ satisfy the three term recurrence relation~\eqref{rrrV},
with
\begin{eqnarray*}
A_m&=&\frac{1}{2\pi i}\,\int_{C}\, zV_{m} (  z ) \, \mathcal{F} ( z ) \, G_{m+1}
( z )  \, dz \, ,\\
B_m&=&\frac{1}{2\pi i}\,\int_{C}zV_{m} (  z )\,  \mathcal{F} ( z )\,  G_{m}
( z ) \, dz \, ,\\
C_m&=&\frac{1}{2\pi i}\,\int_{C}zV_{m} (  z ) \, \mathcal{F} ( z ) \, G_{m-1}
( z ) \, dz \, .
\end{eqnarray*}

For our purposes in this work we need the following definition:
\begin{defi}
Let $ \, {\mathcal{U}}$ be a vector of linear functionals. We denote by
$\widehat{{\mathcal{U}}}$, 
the {\it nor\-ma\-li\-zed vector of linear functionals}
associated with ${\mathcal{U}}$, i.e.,
$$ \widehat{{\mathcal{U}}}=(({\mathcal{U}}({\mathcal{P}}_0))^{-1})^T {\mathcal{U}} \, ,
\mbox{ where } \ \  {\mathcal{P}}_0(x) = \left[
\begin{matrix}
1 & x
\end{matrix}
\right]^T \, . 
$$
\end{defi}

Furthermore, from this definition we have
\begin{equation*}
\widehat{{\mathcal{U}}}({\mathcal{P}}_0) = ((({\mathcal{U}}({\mathcal{P}}_0))^{-1})^T
{\mathcal{U}})({\mathcal{P}}_0) = {\mathcal{U}}({\mathcal{P}}_0)({\mathcal{U}}({\mathcal{P}}_0))^{-1} = I_{2 \times 2}.
\end{equation*}

From now on and in the next section we consider a normalized vector of linear functionals (we set ${\mathcal{U}}={\widehat{\mathcal{U}}}$).


The next theorem shows how the generalized Markov function is directly related with the Weyl function associated with the block tridiagonal Jacobi matrix~\eqref{J}.
\begin{teo}\label{teo2} 
Let ${\mathcal{U}}$ be a normalized vector of linear functionals, ${\mathcal{F}}$ the generalized Markov function, and $R_J$ the Weyl function associated with the Jacobi block matrix, $J$, defined by~\eqref{J}. Then, we have that
$$R_J(z)=M\,{\mathcal{F}}(z)\,M^{-1} \, , $$
where $M=\left[
            \begin{matrix}
              1 & 0 \\
              -a_1 & 1 
            \end{matrix}
          \right]
\, $.
\end{teo}
\begin{proof}
Let $\{ \mathcal{B}_m \}$ be a vector sequence of polynomials left-orthogonal with respect to $\mathcal{U}$.
To determine the value of $(e_0^T J^n e_0),n \in \, {\mathbb{N}}$ we consider the following matricial identity
$$J\left[%
\begin{matrix}
  {{\mathcal{B}}}_0(x) \\
  \vdots \\
  {{\mathcal{B}}}_m(x)\\
  \vdots 
\end{matrix}
\right]=x^2\left[%
\begin{matrix}
  {{\mathcal{B}}}_0(x) \\
  \vdots \\
  {{\mathcal{B}}}_m(x)\\
  \vdots 
\end{matrix}\right] \, , $$
from which we obtain
\begin{eqnarray} \label{fran}J^n\left[%
\begin{matrix}
  {{\mathcal{B}}}_0(x) \\
  \vdots \\
  {{\mathcal{B}}}_m(x)\\
  \vdots 
\end{matrix}
\right]=x^{2n}\left[%
\begin{matrix}
  {{\mathcal{B}}}_0(x) \\
  \vdots \\
  {{\mathcal{B}}}_m(x)\\
  \vdots \\
\end{matrix}\right],\, m \in {\mathbb{N}} \, .
\end{eqnarray}
Hence, from the first equation of the relation~\eqref{fran}
we have that
$$e_0^T J^n e_0 \, {{\mathcal{B}}}_0 +\cdots =x^{2n}
{{\mathcal{B}}}_0 \, .$$

Applying the vector of linear functionals ${\mathcal{U}}$ to the last relation and considering the or\-tho\-go\-na\-li\-ty conditions we have that
$$e_0^T J^n e_0 \, {\mathcal{U}} ({{\mathcal{B}}}_0) =(x^{2n}
{\mathcal{U}} )({{\mathcal{B}}}_0) \, , $$ i.e., $$e_0^T
J^n e_0=(x^{2n}
{\mathcal{U}})({{\mathcal{B}}}_0)({\mathcal{U}} ({{\mathcal{B}}}_0))^{-1} \, . $$

But, from the initial conditions of the three term recurrence relation~\eqref{eq:300} we have
${{\mathcal{B}}}_0= M{\mathcal{P}}_0$, then we obtain
$$e_0^T J^n e_0= M(x^{2n}
{\mathcal{U}})({\mathcal{P}}_0)({\mathcal{U}}({\mathcal{P}}_0))^{-1}M^{-1} \, . $$
Therefore,
$$R_J (z)= M\sum_{n=0}^\infty \frac{(x^{2n}
{\mathcal{U}})({\mathcal{P}}_0)({\mathcal{U}}({\mathcal{P}}_0))^{-1}}{z^{n+1}}M^{-1} \, . $$
Since we take $ \, {\mathcal{U}}$ a normalized vector functional we get the desired representation for
$R_J$.
\end{proof}


\section{Main Result} \label{sec:3}

To establish the relation between the solutions of an integrable system and the vector of polynomials ${\mathcal{B}}_{m}$
we have to recall that in this case $\mathcal{U}=\mathcal{U}(t)$ depends on $t$ and then, it is possible to define the derivative of ${\mathcal{U}}$ (cf.~\cite{Vilenkin}) as usual:
$$\frac{d\,\mathcal{U}}{dt}:{\mathbb{P}}^2\rightarrow \mathcal{M}_{2 \times 2}(\mathbb{C})$$
such that, for each $\mathcal{P} \in {\mathbb{P}}^2  $,
$$\frac{d\,{\mathcal{U}}}{dt}(\mathcal{P})=\lim_{\Delta t \rightarrow 0}\frac{{\mathcal{U}}\{t+\Delta t\}(\mathcal{P})-{\mathcal{U}}\{t\}(\mathcal{P})}{\Delta t} \, . $$

Obviously, the usual properties for this kind of operators are verified. In particular,
\begin{equation}\label{derivadat}
\frac{d}{dt}({\mathcal{U}}(\mathcal{P}))=\frac{d{\mathcal{U}}}{dt}(\mathcal{P})+{\mathcal{U}}(\dot{\mathcal{P}}), \,\,
\forall \, {\mathcal{P}}\in {\mathbb{P}}^2 \, .
\end{equation}

Now, we establish and prove our main result about the relation between the solutions of an integrable system and the matrix polynomials, $\{ V_m \}$, orthogonal with respect to the generalized Markov function, $\mathcal{F}$, associated with $J$ by the previous theorem:
\begin{teo}\label{teo:7}
Assume that the sequence $\{a_n,b_n,c_n,d_n\},\,\, n \in {\mathbb{N}}$, is uniformly bounded, i.e., $\exists \, K \in {\mathbb{R}}_+$ such that $\operatorname{max}\{|a_n(t)| , |b_n(t)| , |c_n(t)| , |d_n(t)| \} \leq M$ for all $n \in {\mathbb{N}}$ and $t \in {\mathbb{R}}$. Also, consider that $\dot{a}_1=c_1$. Then, the following conditions are equivalent:
\begin{itemize}
\item[(a)] $\{a_n,b_n,c_n,d_n\},\,\, n \in {\mathbb{N}}$, is a solution of~\eqref{sisteqdif}, this is,
\begin{equation}\label{eq:Jlax}
\dot{J}=JJ_{-}-J_{-}J \, .
\end{equation}
\item[(b)] For each $n \in {\mathbb{N}}\cup \{0\}$ we have that
\begin{equation}\label{eq:dif}
\frac{d}{dt}J^n_{11}=J^{n+1}_{11}-J^n_{11}J_{11}+J^n_{11}(J_{-})_{11}-(J_{-})_{11}J^n_{11} \, .
\end{equation}
\item[(c)] For $n \in {\mathbb{N}}$ we have that
\begin{equation}\label{momentos}
\dot{{\mathcal{U}}}_n={\mathcal{U}}_{n+1} - {\mathcal{U}}_n \, {\mathcal{U}}_1 \, .
\end{equation}
\item[(d)] For all $z\in{}\mathbb{C}$ such that $|z|>||J||$,
\begin{equation}\label{eqdifF}
\dot{{\mathcal{F}}}(z)={\mathcal{F}}(z)(zI_{2 \times 2}-\mathcal{U}_1) - I_{2\times 2} \, .
\end{equation}
\item[(e)] For all ${\mathcal{B}} \in {\mathbb{P}}^2$ we have that
\begin{equation}\label{eqdifU}
\left(\frac{d}{dt}{\mathcal{U}}\right)({\mathcal{B}})={\mathcal{U}}(x^2{\mathcal{B}})-{\mathcal{U}}({\mathcal{B}}){\mathcal{U}}_1 \, .
\end{equation}
\item[(f)] For all $m\in \mathbb{N}\cup\{0\}$, the polynomial ${{\mathcal{B}}}_m$ defined by~\eqref{eq:300} satisfies
\begin{equation}\label{CarB}
\dot{{{\mathcal{B}}}}_m(x)=-C_m{{\mathcal{B}}}_{m-1}(x)-D_m{{\mathcal{B}}}_m(x) \, , \ \mbox{ with } \ \
D_m=\left[
        \begin{matrix}
          0 & 0 \\
          c_{2m+1} & 0 \\
        \end{matrix}
      \right] \, .
\end{equation}

\item[(g)] The polynomial $ \{ V_m \} $ defined by~$\mathcal{B}_m (x) = V_m (x^2) \mathcal{P}_0 $ satisfies
\begin{equation}\label{CarV}
\dot{V}_m (x) = - C_m {V}_{m-1}(x) - D_m{V}_m(x) \, ,  \ \ m \in \mathbb{N}\cup\{0\} \, ,
\end{equation}
with $D_m$ given in (f).
\end{itemize}
\end{teo}
\begin{proof}
We will prove this theorem according to the following scheme:
$$(a) \Rightarrow (b) \Rightarrow (c) \Rightarrow (d) \Rightarrow (e) \Rightarrow (f)  \Rightarrow (g) \Rightarrow (a) \, . $$

We start, proving that $(a)\Rightarrow (b)$. Since the sequence $\{a_n,b_n,c_n,d_n\},\,\, n \in {\mathbb{N}}$ is a solution of~\eqref{sisteqdif}
we have that $\dot{J}=[J,J_{-}] $.
It is easy to prove by induction that
$$\frac{d}{dt}J^n=J^nJ_{-}-J_{-}J^n \, .$$
For $n=1$ the result is straightforward.
Now, suppose that for $n=2,\ldots, p$ the relation $\displaystyle \frac{d}{dt}J^p=J^pJ_{-}-J_{-}J^p$ holds. Then, we just have to prove that the results is also valid for $n=p+1$.
Since,
\begin{eqnarray*}
\frac{d}{dt}J^{p+1}&=&\frac{d}{dt}(J^pJ)=\frac{d}{dt}(J^p)J+J^p\frac{d}{dt}(J)\\&=&
(J^pJ_{-}-J_{-}J^p)J+J^p(JJ_{-}-J_{-}J)= J^{p+1}J_{-}-J_{-}J^{p+1},
\end{eqnarray*}
the result holds.
In particular, we have that
\begin{equation*}
\frac{d}{dt}J^n_{11}=(J^nJ_{-})_{11}-(J_{-}J^n)_{11} \, .
\end{equation*}
From~\eqref{eq:J} and~\eqref{blockmatrix},
$$\frac{d}{dt}J^n_{11}=J_{11}^n(J_{-})_{11}+J_{12}^n(J_{-})_{21}-(J_{-})_{11}J^n_{11} \, . $$
But, we have that
$$J^{n+1}_{11}=J^n_{11}J_{11}+J^n_{12}J_{21} \, , $$
or equivalently, since $J_{21}=(J_{-})_{21}$, that
$$J^n_{12}(J_{-})_{21}=J^{n+1}_{11}-J^n_{11}J_{11} \, . $$
Therefore, we have that
$$\frac{d}{dt}J^n_{11}=J^{n+1}_{11}-J^n_{11}J_{11}+[J^n_{11},(J_{-})_{11}] \, . $$

To prove $(b)\Rightarrow (c)$ we have to read the differential equation~\eqref{eq:dif} in terms of the moments.
Regarding that we are working with normalized vector functionals and that $J^n_{11}=M(x^{2n}{\mathcal{U}})({\mathcal{P}}_0)M^{-1}$ (cf. Theorem~\ref{teo2}) the relation~\eqref{eq:dif} becomes
\begin{multline*}
\frac{d}{dt}(M(x^{2n}
{\mathcal{U}})({\mathcal{P}}_0)M^{-1})=M(x^{2(n+1)}
{\mathcal{U}})({\mathcal{P}}_0)M^{-1}-M(x^{2n}
{\mathcal{U}})({\mathcal{P}}_0)M^{-1}J_{11}\\+M(x^{2n}
{\mathcal{U}})({\mathcal{P}}_0)M^{-1}(J_{-})_{11}-(J_{-})_{11}M(x^{2n}
{\mathcal{U}})({\mathcal{P}}_0)M^{-1}.
\end{multline*}
Since the moment ${\mathcal{U}}_n={\mathcal{U}}(\mathcal{P}_n)=(x^{2n}
{\mathcal{U}})({\mathcal{P}}_0)$ we have that
\begin{multline}\label{aux}
\frac{d}{dt}(M{\mathcal{U}}_nM^{-1})=M{\mathcal{U}}_{n+1}M^{-1}-M{\mathcal{U}}_nM^{-1}J_{11}\\
+M{\mathcal{U}}_nM^{-1}(J_{-})_{11}-(J_{-})_{11}M{\mathcal{U}}_nM^{-1}.
\end{multline}
Taking in consideration that
$$\dot{M}=\left[
            \begin{matrix}
              0 & 0 \\
              -\dot{a_1} & 0 \\
            \end{matrix}
          \right], \,\,   \dot{\widehat{M^{-1}}}=\left[
            \begin{matrix}
              0 & 0 \\
              \dot{a_1} & 0 \\
            \end{matrix}
          \right],\,\,(J_{-})_{11}=\left[
            \begin{matrix}
              0 & 0 \\
              c_1 & 0 \\
            \end{matrix}
          \right], \,\, \dot{a_1}=c_1 \, , $$
and the fact that~\eqref{aux} can be written like
$${\dot{\mathcal{U}}}_n={\mathcal{U}}_{n+1}-M^{-1}[\dot{M}+(J_{-})_{11}M]{\mathcal{U}}_n+{\mathcal{U}}_n[-M^{-1}J_{11}M- \dot{\widehat{M^{-1}}}M+M^{-1}(J_{-})_{11}]$$
we have
$${\dot{\mathcal{U}}}_n={\mathcal{U}}_{n+1}-{\mathcal{U}}_n \, M^{-1}J_{11}M \, . $$
Since $(x^{2n}{\mathcal{U}})({\mathcal{P}}_0)=M^{-1}J^n_{11}M$ we have that the last relation is equivalent to~\eqref{momentos}.
Now, we prove that $(c)\Rightarrow (d)$ remember that ${\mathcal{F}}$ can be written
$${\mathcal{F}}(z)= \sum_{n=0}^\infty \frac{{\mathcal{U}}_n}{z^{n+1}} \, . $$

Then, from~\eqref{momentos},
\begin{eqnarray*}
\frac{d}{dt}{\mathcal{F}}(z)&=&\sum_{n=0}^\infty \frac{\dot{{\mathcal{U}}}_n}{z^{n+1}}\\
&=&\sum_{n=0}^\infty \frac{{\mathcal{U}}_{n+1}}{z^{n+1}}-\left(\sum_{n=0}^\infty \frac{{\mathcal{U}}_{n}}{z^{n+1}}\right){\mathcal{U}}_1\\
&=&z\sum_{n=0}^\infty \frac{{\mathcal{U}}_{n+1}}{z^{n+2}}-{\mathcal{F}}(z){\mathcal{U}}_1
\\
&=&z\left({\mathcal{F}}(z)-\frac{\mathcal{U}_0}{z}\right)-{\mathcal{F}}(z){\mathcal{U}}_1
\\&=&{\mathcal{F}}(z)(zI_{2\times 2}-{\mathcal{U}}_1)-{\mathcal{U}}_0 \, ,
\end{eqnarray*}
and as ${\mathcal{U}}_0 = I_{2 \times 2}$ we get the desired equation for $\mathcal{F}$.

To prove that $(d)\Rightarrow(e)$ we are going to obtain the derivative of the vector functional ${\mathcal{U}}$ from~\eqref{eqdifF}. To do this, we use the linearity of ${\mathcal{U}}$ and the convergence of the series,
\begin{equation} \label{expaF}
{\mathcal{F}}(z)= \sum_{n=0}^\infty
\frac{{\mathcal{U}}\left( x^{2n} {\mathcal{P}_0}(x) \right) }{z^{n+1}}  = {\mathcal{U}}_x\left(
\frac{{\mathcal{P}_0}(x)}{z-x^2}\right) \, , \ \  ||J||<|z| \, .
\end{equation}

For sake of simplicity here and in the next expressions we will use ${\mathcal{U}}_x={\mathcal{U}}$. From~\eqref{eqdifF} and~\eqref{expaF},
\begin{align*}
&\frac{d}{dt}{\mathcal{U}}\left(
\frac{{\mathcal{P}_0}}{z-x^2}\right)={\mathcal{U}}\left(
\frac{{\mathcal{P}_0}}{z-x^2}\right)(zI_{2 \times 2}-{\mathcal{U}}_1)-{\mathcal{U}}_0\\
&=  {\mathcal{U}}\left(
\left(1+\frac{x^2}{z-x^2}\right){\mathcal{P}_0}\right)-{\mathcal{U}}\left(
\frac{{\mathcal{P}_0}}{z-x^2}\right){\mathcal{U}}(x^2{\mathcal{P}}_0) - {\mathcal{U}}({\mathcal{P}}_0)\\
&= {\mathcal{U}}\left(
\frac{x^2}{z-x^2}{\mathcal{P}_0}\right)-{\mathcal{U}}\left(
\frac{{\mathcal{P}_0}}{z-x^2}\right){\mathcal{U}}(x^2{\mathcal{P}}_0).
\end{align*}

Now, we define de auxiliary vector functionals ${\mathcal{U}}^1,\,{\mathcal{U}}^2:{\mathbb{P}}^2\rightarrow {\mathcal{M}}_{2 \times 2} (\mathbb{C})$ such that:
$$
  \begin{cases}
    {\mathcal{U}}^1={\mathcal{U}}(x^2\mathcal{B}) \\
    {\mathcal{U}}^2={\mathcal{U}}(\mathcal{B}){\mathcal{U}}(x^2\mathcal{P}_0)
  \end{cases}
$$
for each ${\mathcal{B}} \in {\mathbb{P}}^2$. We remark that $\frac{1}{z-x^2}{\mathcal{P}}_0$ do not depend on $t \in {\mathbb{R}}$.
 In~\eqref{expaF}, denoting ${\dot{\mathcal{U}}}=\frac{d}{dt}{\mathcal{U}}$, we have that 
 $ {\mathcal{U}} = {\mathcal{U}}^1 - {\mathcal{U}}^2$ 
 over $\frac{1}{z-x^2}{\mathcal{P}}_0$, being
$${\mathcal{U}}\left(\frac{1}{z-x^2}{\mathcal{P}}_0\right)=\frac{1}{z}{\mathcal{U}}({\mathcal{P}}_0)+
\frac{1}{z^2}{\mathcal{U}}(x^2{\mathcal{P}}_0)+\ldots, \,\,\, |z|>||J|| \, . $$
Hence, we have ${\mathcal{U}}={\mathcal{U}}^1-{\mathcal{U}}^2$ over ${\mathbb{P}}^2$, this is, we have~\eqref{eqdifU}.

For proving that $(e)\Rightarrow(f)$ we use the fact that $\dot{{{\mathcal{B}}}}_m$ can be written in terms of the sequence $\{{\mathcal{B}}_m\}
$,
\begin{equation}\label{greve}
\dot{{{\mathcal{B}}}}_m=\sum_{j=0}^m \alpha_j^m {{\mathcal{B}}}_j \, .
\end{equation}

For $m=0,1$, the above expression is
\begin{equation}\label{eqdifB}
\dot{{{\mathcal{B}}}}_m(x)=\alpha_{m-1}^m {{\mathcal{B}}}_{m-1}(x) +\alpha_m^m {{\mathcal{B}}}_m(x) \, .
\end{equation}

Let $m \geq 2$ be fixed. We are going to show that~\eqref{eqdifB} holds, also for $m$. Due to the orthogonality conditions~\eqref{cotra1}, i.e., $(x^{2k} {\mathcal{U}}) \left( {\mathcal{B}}_m \right)=
\Delta_m \delta_{k,m}, \, k=0,\ldots, m-1, \, m \in {\mathbb{N},}$ from~\eqref{greve}, we have that
$${\mathcal{U}}(\dot{{{\mathcal{B}}}}_m)=\alpha_{0}^m {\mathcal{U}}({{\mathcal{B}}}_{0}) \, . $$

In fact, using~\eqref{derivadat} and~\eqref{eqdifU},
\begin{eqnarray*}
0_{2 \times 2}= \frac{d}{dt}({\mathcal{U}}({{\mathcal{B}}}_m))&=& {\dot{\mathcal{U}}}({{\mathcal{B}}}_m)+{\mathcal{U}}(\dot{{{\mathcal{B}}}}_m)\\&=&
{\mathcal{U}}(x^2{{\mathcal{B}}}_m)-{\mathcal{U}}({{\mathcal{B}}}_m){\mathcal{U}}(x^2{\mathcal{P}_0})+\alpha_{0}^m {\mathcal{U}}({{\mathcal{B}}}_{0}).
\end{eqnarray*}
Thus, from orthogonality, we have $\alpha^m_0=0_{2 \times 2}$. We proceed by induction on $m$, assuming
$$\alpha^m_0=\alpha^m_1=\cdots=\alpha^m_{j-1}=0_{2 \times 2} \, , $$
for a fixed $j < m-1$. Using~\eqref{greve} and, again,~\eqref{eqdifU} and the orthogonality conditions,
\begin{eqnarray*}
0_{2 \times 2}= \frac{d}{dt}({\mathcal{U}}(x^{2j}{{\mathcal{B}}}_m))&=& {\dot{\mathcal{U}}}(x^{2j}{{\mathcal{B}}}_m)+{\mathcal{U}}(x^{2j}\dot{{{\mathcal{B}}}}_m)\\&=&
{\mathcal{U}}(x^{2j+1}{{\mathcal{B}}}_m)-{\mathcal{U}}(x^{2j}{{\mathcal{B}}}_m){\mathcal{U}}(x^2{\mathcal{P}_0})+\alpha_{j}^m {\mathcal{U}}(x^{2j}{{\mathcal{B}}}_{j})\\&=&\alpha_{j}^m {\mathcal{U}}(x^{2j}{{\mathcal{B}}}_{j})\\&=&\alpha_{j}^m \Delta_j,
\end{eqnarray*}
where $\Delta_j$ is an invertible matrix. Thus, $\alpha_j^m=0_{2 \times 2}$ and~\eqref{eqdifB} is verified for any $m \in {\mathbb{N}}$.

Our next purpose is to determine $\alpha^m_m$ and $\alpha_{m-1}^m$. From,~\eqref{eqdifB},
we have that
$${\mathcal{U}}(x^{2(m-1)}\dot{{{\mathcal{B}}}}_m)=\alpha_{m-1}^m {\mathcal{U}}(x^{2(m-1)} {{\mathcal{B}}}_{m-1}).
$$
Then, because of~\eqref{eqdifU} and the orthogonality conditions
\begin{eqnarray*}
0_{2 \times 2}&=&\frac{d}{dt}({\mathcal{U}}(x^{2(m-1)}{{\mathcal{B}}}_m))\\&=&
{\mathcal{U}}(x^{2m}{{\mathcal{B}}}_m)-{\mathcal{U}}(x^{2(m-1)}{{\mathcal{B}}}_m){\mathcal{U}}(x^2{\mathcal{P}_0})
+\alpha_{m-1}^m {\mathcal{U}}(x^{2(m-1)}{{\mathcal{B}}}_{m-1}).
\end{eqnarray*}
Therefore $\alpha_{m-1}^m=-\Delta_m(\Delta_{m-1})^{-1}=-C_m$.

On the other hand, writing
\begin{equation}\label{desenB}
{{\mathcal{B}}}_m(x)=\sum_{j=0}^m \beta_j^m {\mathcal{P}}_j(x) \, ,
\end{equation}
and comparing the coefficient of $x^{2m}$ and $x^{2m+1}$ in both sides of~\eqref{desenB}, we obtain
$$
\beta_j^m=\left[
              \begin{matrix}
                1 & 0 \\
                \beta_m & 1 \\
              \end{matrix}
            \right] \, .
$$
Moreover, taking derivatives in~\eqref{desenB} and comparing with~\eqref{eqdifB}, we see that~$\dot{\beta}_j^m$ $=\alpha_m^m$ or, what is the same,
$$\alpha_m^m=\left[
              \begin{matrix}
                0 & 0 \\
                \alpha_m & 0 \\
              \end{matrix}
            \right] \, ,
$$
where we need to determine $\alpha_m$. From~\eqref{eqdifB} and~\eqref{eqdifU},
\begin{multline*}
\alpha_m^m {\mathcal{U}}(x^{2m}{{\mathcal{B}}}_m)=\frac{d}{dt}{\mathcal{U}}(x^{2m}{{\mathcal{B}}}_m)-{\mathcal{U}}(x^{2(m+1)}{{\mathcal{B}}}_m)
\\+{\mathcal{U}}(x^{2m}{{\mathcal{B}}}_m){\mathcal{U}}(x^2{\mathcal{P}_0})+C_m {\mathcal{U}}(x^{2m}{{\mathcal{B}}}_{m-1}) \, .
\end{multline*}
Using the orthogonality conditions and~\eqref{eq:300}
\begin{eqnarray*}
\alpha_m^m \Delta_m=\left(\frac{d}{dt}(\Delta_m)-B_m \Delta_m
+\Delta_m{\mathcal{U}}(x^2{\mathcal{P}_0})\right) \, .
\end{eqnarray*}
Thus,
\begin{eqnarray*}
\alpha_m^m + B_m=\frac{d}{dt}(\Delta_m)(\Delta_m)^{-1}+\Delta_m(M^{-1}J_{11}M)(\Delta_m)^{-1} \, .
\end{eqnarray*}
But, $\Delta_m =C_mC_{m-1}\cdots C_1 \Delta_0$ with $\Delta_0=M$ (see~\cite{anamendes}). 
Then,
\begin{multline}
\label{eq:200}
\alpha_m^m + B_m=\frac{d}{dt}(C_m C_{m-1}\cdots C_1)(C_m C_{m-1}\cdots C_1)^{-1}\\+(C_mC_{m-1}\cdots C_1)( \dot{M}M^{-1} +J_{11})(C_mC_{m-1}\cdots C_1)^{-1} \, .
\end{multline}

The matrix $C_mC_{m-1}\cdots C_1$ is an upper triangular matrix. Moreover, because of $\dot{a}_1 = c_1$ also
$\dot{M}M^{-1} +J_{11}$ is an upper triangular matrix and, then, the matrix in the left-side of~\eqref{eq:200} is upper triangular and, consequently $\alpha_m=c_{2m+1}$.

Now, (g) is true as~\eqref{CarV} is the interpretation of~\eqref{CarB} for the polynomials $\{ V_m \}$.

 Finally, to prove that $(g)\Rightarrow(a)$ we have to take derivatives in~\eqref{rrrV},
\begin{multline*}
 x \dot{V}_m (x) = \dot{A}_{m} {V}_{m+1} (x) + A_{m} \dot{V}_{m+1} (x)  + \dot{B}_m
{V}_m (x) \\
+  B_m \dot{V}_m (x)+ \dot{C}_m {V}_{m-1} (x) + C_m \dot{V}_{m-1} (x), \, \, m \geq 1 \, ,
 \end{multline*}
Using~\eqref{CarV} and taking into account~\eqref{rrrV} we get
\begin{multline*}
(A_mC_{m+1} - C_mA_{m-1} - D_mB_m + B_mD_m){V}_{m} (x)  \\
+(B_m C_m - C_mB_{m-1} - D_mC_m + C_m D_{m-1}){V}_{m-1} (x)  
 \\
+ (A_mD_{m+1} - D_mA_m){V}_{m+1} (x) = \dot{A}_m {V}_{m+1} (x) + \dot{B}_m {V}_{m} (x) + \dot{C}_m {V}_{m-1} (x) \, .
\end{multline*}
Hence, we arrive to
\begin{equation}\label{sistmatri}
\begin{cases}
     \dot{A}_m=A_mD_{m+1}-D_mA_m\\
     \dot{B}_m=A_mC_{m+1}-C_mA_{m-1}+B_mD_m-D_mB_m\\
     \dot{C}_m=B_m C_m-C_mB_{m-1}+C_m D_{m-1}-D_mC_m
  \end{cases} \, , \ \ m=0,1, \ldots \, .
\end{equation}
Taking into account that, with the above notation, $D_m=(J_{-})_{m+1,m+1}$, we see that~\eqref{sistmatri} is equivalent to~\eqref{eq:Jlax}.
\end{proof}

\begin{rem}
We shall notice that the equation~\eqref{CarV} for the matrix polynomials, $\{ V_m \}$, appears in the study of 
semi-classic families of matrix orthogonal polynomials done in~\cite{cmv} and~\cite{DurIsmail}.
\end{rem}

\section{Representation for the Weyl function} \label{sec:4}

In this section we present a result that gives a explicit expression for the Weyl function and we also present another result that gives, under some conditions, a representation of the vector of linear functionals associated with the system studied in the last section.

Considering that $\displaystyle e^{x^2t}=\sum_{k=0}^{+\infty} \frac{t^k}{k!}x^{2k}$ and given a vector of linear functionals ${\mathcal{U}^0}:{\mathcal{P}}\rightarrow {\mathcal{M}}_{2 \times 2}(\mathbb{R})$, which is the vector of functionals $\mathcal{U}$ for $t = 0$, we can always define a vector of linear functionals $e^{x^2t}{\mathcal{U}^0} : {\mathcal{P}}\rightarrow {\mathcal{M}}_{2 \times 2}(\mathbb{R})$ such as
$$
(e^{x^2t}{\mathcal{U}}^0)({\mathcal{P}}_j) = \left(\sum_{k=0}^{+\infty} \frac{t^k}{k!}x^{2k}
 {\mathcal{U}}^0 
\right)(x^{2j}{\mathcal{P}}_0)=
\sum_{k=0}^{+\infty} \frac{t^k}{k!}{\mathcal{U}}^0 (x^{2(j+k)}\mathcal{P}_{0}) \, . 
$$
Now we give a representation of $\mathcal{U}$ associated with the problem under discussion.
\begin{teo} 
In the conditions of Theorem~\ref{teo:7} assume that the vector of linear func\-tio\-nals ${\mathcal{U}}$ verifies
${\mathcal{U}}({\mathcal{P}})=(e^{x^2t}{\mathcal{U}}^0)({\mathcal{P}})E $, for some $E\in {\mathcal{M}}_{2 \times 2}(\mathbb{C})$ 
Then, $\{a_n,b_n,c_n,d_n\}$, $n \in {\mathbb{N}}$, is a solution of~\eqref{sisteqdif}.
\end{teo}

\begin{proof}
Since ${\mathcal{U}}$ is a normalized vector functional necessarily the assumption ${\mathcal{U}}({\mathcal{P}})=(e^{x^2t}{\mathcal{U}}^0)({\mathcal{P}})E$ implies
$$E=[(e^{x^2t}{\mathcal{U}}^0)({\mathcal{P}}_0)]^{-1} \, . $$
On the other hand, with this assumption, if we want to prove that $\{ a_n,b_n,c_n,$ $ d_n \}$, $n \in {\mathbb{N}}$ is a solution of~\eqref{sisteqdif} it is sufficient to show that~\eqref{momentos} holds.

Since, 
$$
\displaystyle \frac{d}{dt}(e^{x^2t}{\mathcal{U}}^0)({\mathcal{P}}_0)=(e^{x^2t}{\mathcal{U}}^0)(x^2{\mathcal{P}}_0)
$$ 
and 
$$
\displaystyle \frac{dE}{dt}=-E\frac{d[(e^{x^2t}{\mathcal{U}}^0)({\mathcal{P}}_0)]}{dt} E = - E ((e^{x^2t}{\mathcal{U}}^0)(x^2{\mathcal{P}}_0))E =
- E \, {\mathcal{U}}(x^2{\mathcal{P}}_0)
$$ 
if we take the derivatives in
$$
{\mathcal{U}}(x^{2k}\mathcal{P}_0)=(e^{x^2t}{\mathcal{U}}^0)(x^{2k}{\mathcal{P}}_0)  E \, ,
$$
we arrive to~\eqref{momentos}.
\end{proof}
\begin{teo}
Assume that the sequence $\{a_n,b_n,c_n,d_n\},\,\, n \in {\mathbb{N}}$, is uniformly bounded, i.e., $ \exists \, K \in {\mathbb{R}}_+$ such that $\operatorname{max}\{ |a_n(t)| , |b_n(t)| , |c_n(t)| , |d_n(t)| \} \leq M$ for all $n \in {\mathbb{N}}$ and $t \in {\mathbb{R}}$.
Then $\{a_n,b_n,c_n,d_n\}$, $n \in {\mathbb{N}}$, is a solution of~\eqref{sisteqdif} if, and only if, the Weyl function satisfy
\begin{equation}\label{eqdif:RJ}
\dot{R}_J(z)=R_J(z)(zI_{2 \times 2}-J_{11})-I_{2 \times 2}+[R_J(z),(J_{-})_{11}] \, ,
\end{equation}
for all $z\in{}\mathbb{C}$ such that $|z|>||J||$.
 \\
Moreover, the Weyl function in this case is given by
\begin{equation} \label{Rexp}
R_J(z)=e^{zt} M \, T (t,z) (N(t))^{-1} \, ,
\end{equation}
where
$$N(t)=\left[
         \begin{matrix}
           e^{\int_0^t b_1 ds} & e^{\int_0^t b_1 ds} \int_{0}^t a_2 \, e^{\int_0^s( b_2-b_1 )dr}ds \\
           0 & e^{\int_0^t b_2 ds} \\
         \end{matrix}
       \right] \, ,
$$
$$T(t,z)=-\int_{0}^t e^{-zs}M^{-1}N(s)ds + M_0^{-1} R_0(z)  \, , $$
here $M_0$ and $R_0(z)$ are, respectively, $M$ and $R_J(z)$ for $t=0$.
\end{teo}

\begin{proof}
The first part of the proof of this result is trivial. From theorem~\ref{teo:7} we know that if $\{a_n,b_n,c_n,d_n\}$, $n \in {\mathbb{N}}$, is a solution of~\eqref{sisteqdif} it is equivalent to say that~\eqref{eq:dif} is verified.
Starting by~\eqref{eq:dif}, to obtain the relation~\eqref{eqdif:RJ} for ${R}_J$ it is sufficient to take derivatives in~\eqref{eq:J11} and to substitute $\dot{J}_{11}^n$ in
$$
\dot{R}_J(z)= \sum_{n=0}^\infty \frac{\dot{J}^n_{11}}{z^{n+1}} \, , \ \ |z|>||J|| \, .
$$
by its equivalent condition~\eqref{eq:dif}.

Reciprocally, if~\eqref{eqdif:RJ} holds, using the fact that $R_J(z)=M{\mathcal{F}}(z)M^{-1}$ and the paragraph $(d)$ of theorem~\ref{teo:7}, we get
$\{a_n,b_n,c_n,d_n\}$, $n \in {\mathbb{N}}$, is a solution of~\eqref{sisteqdif}.

Moreover, it is easy to see that $M$, $T(t,z)$ and $N(t)$ are, respectively, the solutions of the following Cauchy problems:
$$
  \begin{cases}
    \dot{X}=-(J_{-})_{11}X \\
    X(0)=M_0 \, ,
  \end{cases}
  \begin{cases}
    \dot{X}=-e^{-zt}M^{-1} N (t) \\
    X(0)=M_0^{-1}R_0(z) \, ,
  \end{cases}
 \mbox{and} \,
  \begin{cases}
    \dot{X}=(J_{11}-(J_{-})_{11})X \\
    X(0)=I_{2 \times 2} \, .
  \end{cases}
$$

Taking derivatives in the right-hand side of~\eqref{Rexp}, and checking the initial conditions, we prove that $R_J$ is a solution of the following Cauchy problem:
\begin{eqnarray}\label{cauchyproblem}
  \begin{cases}
    \dot{X}=X(zI_{2 \times 2}-J_{11})-I_{2 \times 2}+[X,(J_{-})_{11}] \\
    X(0)=R_0(z)
  \end{cases} \, ,
\end{eqnarray}
From~\cite{Vilenkin}, we know that~\eqref{cauchyproblem} has a unique solution. On the other hand, from~\eqref{eqdif:RJ} we have that $R_J$ is a solution of~\eqref{cauchyproblem}. Then, we arrive to \eqref{Rexp}.
\end{proof}

\ifx\undefined\bysame
\newcommand{\bysame}{\leavevmode\hbox to3em{\hrulefill}\,}
\fi

\end{document}